\documentclass{amsart}
\usepackage{amssymb,amstext,amsmath,amscd,amsthm,amsfonts,enumerate,graphicx,latexsym,stmaryrd}
\usepackage[usenames]{color}
\usepackage[all]{xy}
\newtheorem{thm}{Theorem}[section]
\newtheorem{lem}[thm]{Lemma}
\newtheorem{prop}[thm]{Proposition}
\newtheorem{cor}[thm]{Corollary}
\theoremstyle{definition}
\newtheorem{dfn}[thm]{Definition}
\newtheorem{ques}[thm]{Question}

\newtheorem{rem}[thm]{Remark}

\newtheorem{ex}[thm]{Example}

\theoremstyle{remark}

\newtheorem*{claim*}{Claim}
\newtheorem*{ac}{Acknowlegments}
\newtheorem*{conv}{Convention}

\numberwithin{equation}{thm}

\def\ass{\operatorname{Ass}}
\def\c{\mathfrak{c}}
\def\cm{\operatorname{\mathsf{CM}}}
\def\depth{\operatorname{depth}}
\def\Ext{\operatorname{Ext}}
\def\ext{\operatorname{\mathsf{ext}}}
\def\fl{\operatorname{\mathsf{fl}}}
\def\ge{\geqslant}
\def\gr{\operatorname{gr}}
\def\height{\operatorname{ht}}
\def\Hom{\operatorname{Hom}}
\def\I{\mathcal{I}}
\def\image{\operatorname{Im}}
\def\kernel{\operatorname{Ker}}
\def\le{\leqslant}
\def\m{\mathfrak{m}}
\def\Max{\operatorname{Max}}
\def\Min{\operatorname{Min}}
\def\mod{\operatorname{\mathsf{mod}}}
\def\N{\mathbb{N}}
\def\p{\mathfrak{p}}
\def\q{\mathfrak{q}}
\def\rank{\operatorname{rank}}
\def\spec{\operatorname{Spec}}
\def\supp{\operatorname{Supp}}
\def\syz{\Omega}
\def\t{\mathrm{t}}
\def\Tor{\operatorname{Tor}}
\def\tf{\operatorname{\mathsf{tf}}}
\def\tr{\operatorname{\mathsf{Tr}}}
\def\V{\operatorname{V}}
\def\X{\mathcal{X}}
\def\Y{\mathcal{Y}}
\begin{document}
\baselineskip 12.44pt
\allowdisplaybreaks
\title{When is a subcategory Serre or torsionfree?}
\author{Kei-ichiro Iima}
\address[Kei-ichoro Iima]{Department of Liberal Studies, National Institute of Technology, Nara College, 22 Yata-cho, Yamatokoriyama, Nara 639-1080, Japan}
\email{iima@libe.nara-k.ac.jp}
\author{Hiroki Matsui}
\address[Hiroki Matsui]{Department of Mathematical Sciences, Faculty of Science and Technology, Tokushima University, 2-1 Minamijosanjima-cho, Tokushima 770-8506, Japan}
\email{hmatsui@tokushima-u.ac.jp}
\author{Kaori Shimada}
\address[Kaori Shimada]{Department of Liberal Studies, National Institute of Technology, Nara College, 22 Yata-cho, Yamatokoriyama, Nara 639-1080, Japan}
\email{shimada@libe.nara-k.ac.jp}
\author{Ryo Takahashi}
\address[Ryo Takahashi]{Graduate School of Mathematics, Nagoya University, Furocho, Chikusaku, Nagoya 464-8602, Japan}
\email{takahashi@math.nagoya-u.ac.jp}
\urladdr{https://www.math.nagoya-u.ac.jp/~takahashi/}
\thanks{The second author was partly supported by JSPS Grant-in-Aid for JSPS Fellows 19J00158. The fourth author was partly supported by JSPS Grant-in-Aid for Scientific Research 19K03443.}
\begin{abstract}
Let $R$ be a commutative noetherian ring.
Denote by $\mod R$ the category of finitely generated $R$-modules.
In the present paper, we first provide various sufficient (and necessary) conditions for a full subcategory of $\mod R$ to be a Serre subcategory, which include several refinements of theorems of Stanley and Wang and of Takahashi with simpler proofs.
Next we consider when an IKE-closed subcategory of $\mod R$ is a torsionfree class.
We investigate certain modules out of which all modules of finite length can be built by taking direct summands and extensions, and then we apply it to show that the IKE-closed subcategories of $\mod R$ are torsionfree classes in the case where $R$ is a certain numerical semigroup ring.
\end{abstract}
\maketitle
\tableofcontents
\section{Introduction}

Let $R$ be a commutative noetherian ring and denote by $\mod R$ the category of finitely generated $R$-modules.
Various full subcategories of $\mod R$ have been studied so far.
Gabriel \cite{G} completely classified the Serre subcategories of $\mod R$ by establishing an explicit one-to-one correspondence between them and the specialization-closed subsets of $\spec R$.
Takahashi \cite{wide} showed that a wide subcategory of $\mod R$ is Serre.
Stanley and Wang \cite{SW} proved that a torsion class and a narrow subcategory of $\mod R$ are Serre, which extends Takahashi's result.
The first main result of the present paper is Theorem \ref{15}, which provides many kinds of sufficient (and necessary) conditions for a given subcategory of $\mod R$ to be Serre.
The following theorem is part of it, which extends the result of Stanley and Wang mentioned above; recall that a full subcategory $\X$ of $\mod R$ is said to be {\em tensor-ideal} if $M\otimes X\in\X$ for all $M\in\mod R$ and $X\in\X$.

\begin{thm}\label{i1}
Let $R$ be a commutative noetherian ring.
Let $\X$ be a tensor-ideal subcategory of $\mod R$ which is closed under direct summands and extensions.
Then $\X$ is a Serre subcategory of $\mod R$.
\end{thm}

The torsionfree classes of $\mod R$ were classified completely by Takahashi \cite{wide}; he constructed an explicit bijection between them and the subsets of $\spec R$.
Following \cite{E}, we say that a full subcategory of $\mod R$ is {\em IKE-closed} if it is closed under images, kernels and extensions.
Evidently, a torsionfree class is an IKE-closed subcategory.
Thus it is natural to ask whether an IKE-closed subcategory of $\mod R$ is torsionfree.
To get answers to this question, we prove the following theorem, which is the same as Theorem \ref{8}.

\begin{thm}\label{i2}
Let $R$ be a Cohen--Macaulay local ring of dimension one with maximal ideal $\m$ and infinite residue field $k$.
Suppose that the associated graded ring $\gr_\m R$ has positive depth (or equivalently, that $\gr_\m R$ is a Cohen--Macaulay ring).
Let $n$ be a positive integer.
Then every $R$-module of finite length can be built out of $R/\m^n$ by taking direct summands and extensions.
\end{thm}

Applying the above theorem to numerical semigroup rings, we obtain the following theorem, which is the combination of Theorems \ref{16} and \ref{17}.
The second assertion of the theorem below provides complete classification of the IKE-closed subcategories and an affirmative answer to the question stated above.

\begin{thm}\label{i3}
Let $k$ be a field.
Let $R=k[\![H]\!]$ be the (completed) numerical semigroup ring, where either
\begin{itemize}
\item
$H=\langle a,b\rangle$ with $a>b>0$ and $\gcd(a,b)=1$, or
\item
$H=\langle a,a+1,\dots,a+r\rangle$ with $a\ge2,\,r\ge1$ and $k$ infinite.
\end{itemize}
Then the following two statements hold true.
\begin{enumerate}[\rm(1)]
\item
Every $R$-module of finite length can be built out of $R/\c$ by taking direct summands and extensions, where $\c$ is the conductor of $R$.
\item
The IKE-closed subcategories of $\mod R$ are the zero subcategory, the subcategory of modules of finite length, the subcategory of torsionfree modules, and $\mod R$ itself.
\end{enumerate}
\end{thm}

We now state the organization of the present paper.
For this, let $R$ be a commutative noetherian ring.
In Section \ref{s2}, we consider when a given subcategory $\X$ of $\mod R$ is a Serre subcategory.
We provide a lot of equivalent conditions for $\X$ to be Serre, which includes Theorem \ref{i1}.
In Section \ref{s3}, we establish the question asking whether an IKE-closed subcategory of $\mod R$ is a torsionfree class.
We give a couple of affirmative answers to the question.
In Section \ref{s4}, we keep exploring the question given in the previous section.
We first reduce it to a more accessible question, and then give positive answers in the case where $R$ is excellent.
In Section \ref{s5}, we explore what modules can be produced by taking only direct summands and extensions.
We find out a way to get a certain extension of modules, and prove Theorem \ref{i2}.
In Section \ref{s6}, we apply results obtained mainly in the previous two sections to the case where $R$ is a numerical semigroup ring, and prove theorems which include Theorem \ref{i3} as a special case.

We close the section by stating our convention.

\begin{conv}
Throughout the present paper, we assume that all rings are commutative and noetherian, that all modules are finitely generated, and that all subcategories are strictly full.
Let $R$ be a (commutative noetherian) ring.
Denote by $\mod R$ the category of (finitely generated) $R$-modules.
The transpose of a matrix $A$ is denoted by ${}^t\!A$.
For a (strictly full) subcategory $\X$ of $\mod R$ we say that:
\begin{enumerate}[(i)]
\item
$\X$ is {\em closed under direct summands} provided if $X\in\X$ and $Y$ is a direct summand of $X$, then $Y\in\X$.
\item
$\X$ is {\em closed under extensions} provided that for an exact sequence $0\to L\to M\to N\to0$ in $\mod R$, if $L$ and $N$ are in $\X$, then $M$ is also in $\X$.
\item
$\X$ is {\em closed under subobjects} (resp. {\em quotients}) provided that if $X$ is a module in $\X$ and $Y$ is a submodule of $X$, then $Y$ (resp. $X/Y$) is also in $\X$.
\item
$\X$ is {\em closed under kernels} (resp. {\em images}, {\em cokernels}) if the kernel (resp. image, cokernel) of each homomorphism $X\to X'$ with $X,X'\in\X$ belongs to $\X$.
\end{enumerate}
\end{conv}

\section{Serre subcategories}\label{s2}

Recall that a {\em Serre} subcategory of $\mod R$ is by definition a subcategory of $\mod R$ which is closed under subobjects, quotients and extensions.
In this section, we investigate when a given subcategory of $\mod R$ is Serre.
Precisely speaking, we shall reduce those defining conditions of a Serre subcategory as much as possible, keeping the resulting conditions defining a Serre subcategory.

For a subcategory $\X$ of $\mod R$, we set $\supp\X=\bigcup_{X\in\X}\supp X$, where $\supp X$ stands for the support of the $R$-module $X$.
For a subset $\Phi$ of $\spec R$ we denote by $\supp^{-1}\Phi$ the subcategory of $\mod R$ consisting of $R$-modules $M$ with $\supp M\subseteq\Phi$.
Note that $\supp\X$ is always a specialization-closed subset of $\spec R$, while $\supp^{-1}\Phi$ is always a Serre subcategory of $\mod R$.
We establish a key lemma.

\begin{lem}\label{10}
Let $\X$ be a subcategory of $\mod R$ closed under extensions.
\begin{enumerate}[\rm(1)]
\item
Put $\Phi=\{\p\in\spec R\mid R/\p\in\X\}$.
One then has $\supp^{-1}\Phi\subseteq\X$.
If $\supp\X\subseteq\Phi$, then $\X$ is Serre.
\item
If $\X\ne\supp^{-1}(\supp\X)$, then there exists $\p\in\supp\X$ with $R/\p\notin\X$ and $\supp^{-1}(\V(\p)\setminus\{\p\})\subseteq\X$.
\item
Let $\p$ be a prime ideal of $R$ such that $\supp^{-1}(\V(\p)\setminus\{\p\})\subseteq\X$.
Suppose that $\X$ is closed under direct summands and contains a module $X$ satisfying $\Min X=\{\p\}$ and $\p X_\p=0$.
Then $R/\p$ belongs to $\X$.
\end{enumerate}
\end{lem}

\begin{proof}
(1) Let $M\in\supp^{-1}\Phi$.
Take a filtration $0=M_0\subseteq\cdots\subseteq M_n=M$ of submodules of $M$ such that for each $i$ one has $M_i/M_{i-1}\cong R/\p_i$ with $\p_i\in\supp M$.
Then $\p_i\in\supp M\subseteq\Phi$ and $R/\p_i\in\X$.
As $\X$ is closed under extensions, we get $M\in\X$.
If $\supp\X\subseteq\Phi$, then $\X=\supp^{-1}\Phi$, and it is Serre.

(2) There is an $R$-module $M$ with $M\notin\X$ and $M\in\Y:=\supp^{-1}(\supp\X)$.
Note that $\Y$ is a Serre subcategory of $\mod R$.
If $M$ is generated by $n$ elements, then there is an exact sequence $0\to M'\to M\to M''\to0$ of $R$-modules such that $M'$ is cyclic, $M''$ is generated by $n-1$ elements, and $M',M''$ are in $\Y$.
By induction on $n$ we may assume that $M$ is cyclic.
Hence the set $\I$ of ideals $I$ of $R$ with $R/I\in\supp^{-1}(\supp\X)$ and $R/I\notin\X$ is nonempty.
As $R$ is noetherian, there exists a maximal element $\p$ of $\I$ with respect to the inclusion relation.
Then $R/\p\notin\X$ and $\V(\p)=\supp R/\p\subseteq\supp\X$.
Take a filtration $0=M_0\subseteq\cdots\subseteq M_n=R/\p$ of submodules of the $R$-module $R/\p$ such that for each $i$ one has $M_i/M_{i-1}\cong R/\p_i$ with $\p_i\in\spec R$.
Then each $\p_i$ contains $\p$.
Assume that the ideal $\p$ of $R$ is not prime.
Then each $\p_i$ must strictly contain $\p$, and the maximality of $\p$ implies $R/\p_i\in\X$ for all $i$,
As $\X$ is closed under extensions, we have $R/\p\in\X$.
This contradiction shows that $\p$ is a prime ideal of $R$.
Let $N\in\supp^{-1}(\V(\p)\setminus\{\p\})$ and $\q\in\supp N$.
Then $\p\subsetneq\q$ and $\supp R/\q\subseteq\supp R/\p\subseteq\supp\X$.
The maximality of $\p$ implies $R/\q\in\X$.
By (1) we get $N\in\X$.
Hence $\supp^{-1}(\V(\p)\setminus\{\p\})\subseteq\X$.

(3) We have $\p\in\ass X\subseteq\supp X=\V(\p)$.
Taking the irredundant primary decomposition of the submodule $0$ of the $R$-module $X$, we find submodules $M,N$ of $X$ with $0=M\cap N$ such that $\ass X/M=\{\p\}$ and every associated prime ideal of $X/N$ strictly contains $\p$.
There is an exact sequence $\sigma:0\to X\to X/M\oplus X/N\to X/(M+N)\to0$.
Note that $\supp X/(M+N)\subseteq\supp X/N\subseteq\V(\p)\setminus\{\p\}$.
We get $X/(M+N)\in\supp^{-1}(\V(\p)\setminus\{\p\})\subseteq\X$ and $(X/M)_\p\cong X_\p\cong\kappa(\p)^{\oplus m}$ for some $m>0$ as $\p X_\p=0$.
Since $\X$ is closed under direct summands and extensions, the exact sequence $\sigma$ shows $X/M\in\X$.
There is an exact sequence $\tau:0\to K\to X/M\to(R/\p)^{\oplus m}\to C\to0$ with $K_\p=C_\p=0$.
As $\ass K\subseteq\ass X/M=\{\p\}$, we have $K=0$.
We also have $C\in\supp^{-1}(\V(\p)\setminus\{\p\})\subseteq\X$.
The exact sequence $\tau$ implies $R/\p\in\X$.
\end{proof}

A subcategory $\X$ of $\mod R$ is said to be {\em $\otimes$-ideal} if $M\otimes_RX\in\X$ for all $M\in\mod R$ and $X\in\X$.
In the proposition below we give a sufficient condition for $\X$ to be Serre.

\begin{prop}\label{1}
Let $\X$ be a $\otimes$-ideal subcategory of $\mod R$ closed under direct summands and extensions.
Then $\X$ is a Serre subcategory of $\mod R$.
\end{prop}

\begin{proof}
Suppose that $\X$ is not Serre.
Then Lemma \ref{10}(2) implies that there exist $X\in\X$ and $\p\in\supp X$ such that $R/\p\notin\X$ and $\supp^{-1}(\V(\p)\setminus\{\p\})\subseteq\X$.
As $\X$ is $\otimes$-ideal, we have $Y:=X/\p X=R/\p\otimes_RX\in\X$.
Note that $\Min Y=\{\p\}$ and $\p Y_\p=0$.
Lemma \ref{10}(3) implies $R/\p\in\X$, which is a contradiction.
\end{proof}

Let $\X$ be a subcategory of $\mod R$.
We say that $\X$ is {\em $\Hom$-ideal} (resp. {\em $\Ext$-ideal}) if $\Hom_R(M,X)$ (resp. $\Ext_R^1(M,X)$) is in $\X$ for all $M\in\mod R$ and $X\in\X$.
We give two sufficient conditions for $\X$ to be Serre.

\begin{prop}\label{11}
Let $\X$ be a $\Hom$-ideal subcategory of $\mod R$ which is closed under direct summands and extensions.
Suppose either {\rm(1)} $\ass\X$ is specialization-closed or {\rm(2)} $\X$ is $\Ext$-ideal.
Then $\X$ is Serre.
\end{prop}

\begin{proof}
Assuming that $\X$ is not Serre, we shall derive a contradiction.
Lemma \ref{10}(2) gives rise to a prime ideal $\p\in\supp\X$ such that $R/\p\notin\X$ and $\supp^{-1}(\V(\p)\setminus\{\p\})\subseteq\X$.

(1) Note that the equality $\ass\X=\supp\X$ holds.
We find a module $X\in\X$ such that $\p\in\ass X$.
Put $Y=\Hom_R(R/\p,X)\in\Hom_R(\mod R,\X)\subseteq\X$.
We have $\p Y_\p=0$ and $\ass Y=\V(\p)\cap\ass X$, the latter of which implies $\Min Y=\{\p\}$.
It follows from Lemma \ref{10}(3) that $R/\p$ is in $\X$, which is a contradiction.

(2) Choose a module $X\in\X$ such that $\p\in\supp X$.
By \cite[Proposition (2.6)]{AB} there is an exact sequence
$$
0\to\Ext_R^1(\tr R/\p,X)\to X/\p X\xrightarrow{\varpi}\Hom_R(\Hom_R(R/\p,R),X)\to\Ext_R^2(\tr R/\p,X)\to0.
$$
Let $Z$ be the image of the map $\varpi$.
By assumption, we have $\Ext_R^1(\tr R/\p,X)\in\Ext_R^1(\mod R,\X)\subseteq\X$.

First, we suppose $Z\in\X$.
Since $\X$ is closed under extensions, we have $Y:=X/\p X\in\X$ and $\p Y_\p=0$.
As $\supp Y=\V(\p)\cap\supp X=\V(\p)$, we get $\Min Y=\{\p\}$.
Lemma \ref{10}(3) yields $R/\p\in\X$, a contradiction.

Next, we suppose $Z\notin\X$.
Then $\supp Z\subseteq\supp X/\p X\subseteq\V(\p)$.
If $Z_\p=0$, then $Z\in\supp^{-1}(\V(\p)\setminus\{\p\})\subseteq\X$, which is contrary to the assumption.
Hence $\p\in\supp Z\subseteq\supp \Hom_R(\Hom_R(R/\p,R),X)$, and we see that $\p\in\ass X$.
Setting $Y=\Hom_R(R/\p,X)$, we have $\p Y_\p=0$, $\ass Y=\V(\p)\cap\ass X$ and $\Min Y=\{\p\}$.
It follows from Lemma \ref{10}(3) that $R/\p$ belongs to $\X$, which is a contradiction.
\end{proof}

A subcategory $\X$ of $\mod R$ is {\em closed under cokernels of monomorphisms} provided for an exact sequence $0\to L\to M\to N\to0$ in $\mod R$ with $L,M\in\X$ one has $N\in\X$.
We collect some elementary facts.

\begin{rem}\label{2}
The following statements hold true for a subcategory $\X$ of $\mod R$.
\begin{enumerate}[(1)]
\item
If $\X$ is closed under finite direct sums and cokernels, then it is $\otimes$-ideal and closed under direct summands.
Indeed, let $M\in\mod R$ and $X\in\X$.
Then there is an exact sequence $R^{\oplus a}\to R^{\oplus b}\to M\to0$, which induces an exact sequence $X^{\oplus a}\to X^{\oplus b}\to M\otimes_RX\to0$.
Hence $M\otimes_RX\in\X$.
Thus $\X$ is $\otimes$-ideal.
That $\X$ is closed under direct summands follows by splicing the split exact sequences $0\to M\to M\oplus N\to N\to0$ and $0\to N\to M\oplus N\to M\to0$ for $R$-modules $M$ and $N$.
\item
If $\X$ is closed under finite direct sums and kernels, then it is $\Hom$-ideal.
Indeed, let $M\in\mod R$ and $X\in\X$.
Then there is an exact sequence $R^{\oplus a}\to R^{\oplus b}\to M\to0$, which induces an exact sequence $0\to\Hom_R(M,X)\to X^{\oplus b}\to X^{\oplus a}$.
Hence $\X$ contains $\Hom_R(\mod R,\X)$.
\item
If $\X$ is $\Hom$-ideal and closed under cokernels of monomorphisms, then it is $\Ext$-ideal.
In fact, let $M\in\mod R$ and $X\in\X$.
Then there is an exact sequence $0\to N\to P\to M\to0$ with $P$ projective, which induces an exact sequence $0\to\Hom_R(M,X)\to\Hom_R(P,X)\to\Hom_R(N,X)\to\Ext_R^1(M,X)\to0$.
\end{enumerate}
\end{rem}

Let $\X$ be a subcategory of $\mod R$ closed under extensions.
Recall that $\X$ is said to be {\em torsion} if it is closed under quotients.
Also, $\X$ is called {\em wide} (resp. {\em narrow}) if it is closed under kernels and cokernels (resp. cokernels).
Now we state and prove the main result of this section, which is the theorem below.
This includes the assertions of the theorems of Stanley and Wang \cite[Theorem 2]{SW} and of Takahashi \cite[Theorem 3.1]{wide} with much simpler proofs.

\begin{thm}\label{15}
For a subcategory of $\X$ of $\mod R$ the following eight conditions are equivalent.\\
{\rm(1)} $\X$ is Serre.\qquad
{\rm(2)} $\X$ is torsion.\qquad
{\rm(3)} $\X$ is wide.\qquad
{\rm(4)} $\X$ is narrow.\\
{\rm(5)} $\X$ is $\otimes$-ideal, and closed under direct summands and extensions.\\
{\rm(6)} $\X$ is $\Hom$-ideal, $\Ext$-ideal, and closed under direct summands and extensions.\\
{\rm(7)} $\X$ is $\Hom$-ideal, and closed under direct summands, extensions and cokernels of monomorphisms.\\
{\rm(8)} $\X$ is $\Hom$-ideal, closed under direct summands and extensions, and $\ass\X$ is specialization-closed.
\end{thm}

\begin{proof}
The implications (1) $\Rightarrow$ (2) $\Rightarrow$ (4) and (1) $\Rightarrow$ (3) $\Rightarrow$ (4) are clear.
The implications (4) $\Rightarrow$ (5) $\Rightarrow$ (1) follow by Remark \ref{2}(1) and Proposition \ref{1}, while (1) $\Rightarrow$ (7) $\Rightarrow$ (6) and (8) $\Rightarrow$ (1) $\Leftarrow$ (6) follow by Remark \ref{2}(2)(3) and Proposition \ref{11}.
It remains to show (1) $\Rightarrow$ (8).
Let $\X$ be Serre.
Then $\X$ is $\Hom$-ideal by Remark \ref{2}(2).
Let $\p\in\ass\X$ and $\q\in\V(\p)$.
There are a monomorphism $R/\p\to X\in\X$ and an epimorphism $R/\p\to R/\q$.
We see $R/\q\in\X$, whence $\q\in\ass\X$.
Thus $\ass\X$ is specialization-closed.
\end{proof}

We say that a subcategory $\X$ of $\mod R$ is {\it $\Tor$-ideal} if $\Tor_1^R(M, X) \in \X$ for all $M \in \mod R$ and $X \in \X$.
To show the final result of this section, we establish a lemma on Ext-ideal and Tor-ideal subcategories.

\begin{lem}\label{torext} 
Let $\X$ be a subcategory of $\mod R$, $X \in \X$, and $\p$ a prime ideal of $R$ of positive grade.
If $\X$ is $\Ext$-ideal, then $X \otimes_R R/\p \in \X$.
If $\X$ is $\Tor$-ideal, then $\Hom_R(R/\p, X) \in \X$.
\end{lem}

\begin{proof}
Let $a_1, \dots, a_n$ be a system of generators of $\p$. 
Taking the $R$-dual of the exact sequence $R^{\oplus n} \xrightarrow{(\begin{smallmatrix}a_1&\cdots&a_n\end{smallmatrix})} R \to R/\p \to 0$ and using the assumption that $\Hom_R(R/\p,R)=0$, we have a free resolution $0 \to R \xrightarrow{{}^\t(\begin{smallmatrix}a_1&\cdots&a_n\end{smallmatrix})} R^{\oplus n} \to C \to 0$ of $C$.
We get two exact sequences
$$
\begin{array}{l}
0 \to \Hom_R(C, X) \to X^{\oplus n} \xrightarrow{(\begin{smallmatrix}a_1&\cdots&a_n\end{smallmatrix})} X \to \Ext_R^1(C, X) \to 0,\\
0 \to \Tor_1^R(C, X) \to X \xrightarrow{{}^\t(\begin{smallmatrix}a_1&\cdots&a_n\end{smallmatrix})} X^{\oplus n} \to C\otimes_R X \to 0.
\end{array}
$$
These exact sequences give rise to isomorphisms $X \otimes_R R/\p \cong \Ext_R^1(C, X)$ and $\Hom_R(R/\p, X) \cong \Tor_1^R(C, X)$, respectively.
This finishes the proof of the lemma.
\end{proof}

The proposition below concerns when a subcategory $\X$ of $\mod R$ with $\supp\X\cap\ass R=\emptyset$ is Serre.

\begin{prop}
Let $\X$ be a subcategory of $\mod R$ closed under direct summands and extensions.
Consider the following conditions on the subcategory $\X$.\\
{\rm(1)} $\X$ is Serre.\qquad
{\rm(2)} $\X$ is $\Ext$-ideal.\qquad
{\rm(3)} $\X$ is $\Tor$-ideal and $\ass \X$ is specialization-closed.\\
Then {\rm(1)} implies {\rm(2)} and {\rm(3)}.
If $\supp \X \cap \ass R =\emptyset$, then those three conditions are equivalent.
\end{prop}

\begin{proof}
Suppose that $\X$ is Serre.
Then $\X$ is $\otimes$-ideal, $\Ext$-ideal and $\ass\X$ is specialization-closed by Theorem \ref{15}.
If $M\in\mod R$ and $X\in\X$, then there is an exact sequence $0\to N\to P\to M\to0$ with $P$ projective, which induces an exact sequence $0\to\Tor_1^R(M,X)\to N\otimes X$.
This shows that $\Tor_1^R(M,X)$ belongs to $\X$, and we see that $\X$ is $\Tor$-ideal.
Therefore, (1) implies (2) and (3).

Assume $\supp \X \cap \ass R =\emptyset$.
Then each $\p \in \supp \X$ satisfies $\V(\p) \cap \ass R = \emptyset$, which means that $\p$ has positive grade by \cite[Exercise 1.2.27]{BH}.
In view of Lemma \ref{torext}, the implication (2) $\Rightarrow$ (1) (resp. (3) $\Rightarrow$ (1)) follows from the same argument as in the proof of Proposition \ref{1} (resp. Proposition \ref{11}(1)).
\end{proof}

\section{A fundamental question on IKE-closed subcategories}\label{s3}

In this section, we introduce the notion of IKE-closed subcategories of $\mod R$, and pose a natural and fundamental question about the comparison of them with torsionfree subcategories of $\mod R$.
We make some observations concerning the question, and give a couple of affirmative answers to the question.

\begin{dfn}
Let $\X$ be a subcategory of $\mod R$.
\begin{enumerate}[\rm(1)]
\item
We say that $\X$ is {\em torsionfree} if it is closed under subobjects and extensions.
This name comes from the fact that the torsionfree $R$-modules form a torsionfree subcategory of $\mod R$.
\item
We say $\X$ is {\it IKE-closed} if it is closed under images, kernels and extensions.
Note that if $\X$ is IKE-closed, then it is closed under direct summands (indeed, $\X$ is closed under direct summands if it is closed under finite direct sums and kernels, by a similar argument to the last part of Remark \ref{2}(1)).
\end{enumerate}
\end{dfn}

The remark below says that the IKE-closed property is preserved under taking factor rings.

\begin{rem}\label{e}
Let $I$ be an ideal of $R$.
Let $\X$ be an IKE-closed subcategory of $\mod R$.
We then denote by $\X\cap\mod R/I$ the subcategory of $\mod R/I$ consisting of $R/I$-modules that belong to $\X$ as $R$-modules.
Then $\X\cap\mod R/I$ is an IKE-closed subcategory of $\mod R/I$.
This is a consequence of the fact that images, kernels and extensions of $R/I$-modules are images, kernels and extensions of $R$-modules, respectively.
\end{rem}

For a subcategory $\X$ of $\mod R$, we set $\ass\X=\bigcup_{X\in\X}\ass X$.
For a subset $\Phi$ of $\spec R$, we denote by $\ass^{-1} \Phi$ the subcategory of $\mod R$ consisting of $R$-modules $M$ with $\ass M \subseteq \Phi$.
By virtue of \cite[Theorem 4.1]{wide}, the assignments $\X\mapsto\ass\X$ and $\Phi\mapsto\ass^{-1}\Phi$ give a one-to-one correspondence between the torsionfree subcategories of $\mod R$ and the subsets of $\spec R$.

Evidently, every torsionfree subcategory of $\mod R$ is IKE-closed.
Thus it is natural to ask the following.

\begin{ques}\label{qIKE}
Is every IKE-closed subcategory of $\mod R$ torsionfree?
Equivalently, does the equality $\X = \ass^{-1}(\ass \X)$ hold for all IKE-closed subcategories $\X$ of $\mod R$\,?
\end{ques}

The following examples indicate that the assumption in the above question is reasonable.

\begin{ex}
\begin{enumerate}[\rm(1)]
\item
Let $(R,\m,k)$ be a local ring which is not a field, and $\X$ the subcategory of $\mod R$ consisting of modules annihilated by $\m$.
Then $\X$ is closed under subobjects, whence it is closed under kernels and images.
However, $\X$ is not closed under extensions.
Indeed, there is an exact sequence $0\to \m/\m^2\to R/\m^2\to R/\m\to0$ and we have $\m/\m^2,R/\m\in\X$ while $R/\m^2\notin\X$.
This example shows that Question \ref{qIKE} has a negative answer if the assumption ``closed under extensions" is extracted.
\item
Let $(R,\m)$ be a Cohen--Macaulay local ring of dimension at least two.
Let $\X$ be the subcategory of $\mod R$ consisting of modules with depth at least two.
Then $\X$ is closed under kernels and extensions by the depth lemma.
This subcategory is not closed under subobjects.
Indeed, $R$ belongs to $\X$ but its submodule $\m$ does not since it has depth one.
This example shows that Question \ref{qIKE} has a negative answer if the assumption ``closed under images" is extracted.
\end{enumerate}
\end{ex}

Here we state three lemmas.
The first one should be well-known to experts, while the second one is used in the next section as well.
The proof of the third one is taken from \cite[Lemma 4.2]{wide}.

\begin{lem}\label{aslem1}
Let $M \in \mod R$.
Let $A, B$ be subsets of $\ass M$ such that $\ass M = A \sqcup B$.
Then there exist $R$-modules $L, N$ with $\ass L = A$ and $\ass N = B$ that fit into an exact sequence $0 \to L \to M \to N \to 0$.
\end{lem}

\begin{proof}
The lemma follows from \cite[Lemma 2.27]{GW}.
As it is written in Japanese, we explain a bit more for the convenience of the reader.
Let $L$ be a maximal element with respect to the inclusion relation of the set of submodules $L'$ of $M$ with $\ass L'\subseteq A$, and put $N=M/L$.
Then $\ass L=A$ and $\ass N=\ass M\setminus A$.
\end{proof}

\begin{lem}\label{remim0}
Let $\X$ be a subcategory of $\mod R$ closed under direct sums and images.
Let $X\in\X$.
Let $M$ be a submodule of $X^{\oplus n}$ with $n\ge0$.
If there is an epimorphism $\pi:X^{\oplus m}\to M$ with $m\ge0$, then $M\in\X$.
\end{lem}

\begin{proof}
Let $\theta:M\to X^{\oplus n}$ be the inclusion map.
Let $f$ be the composite map $\theta\pi:X^{\oplus m}\to X^{\oplus n}$.
We then have $\image f=M$.
Since $\X$ is closed under direct sums and images, the module $M$ belongs to $\X$.
\end{proof}

\begin{lem}\label{18}
Let $\X$ be a subcategory of $\mod R$ closed under images and extensions.
Let $\p$ be a prime ideal of $R$, and let $M$ be an $R$-module.
If $R/\p$ belongs to $\X$ and $\ass M = \{\p\}$, then $M$ belongs to $\X$.	
\end{lem}

\begin{proof}
Assume $M \not\in \X$.
We want to make a filtration $\cdots \subseteq M_n \subseteq M_{n-1} \subseteq \cdots  \subseteq M_0= M$ of submodules of $M$ with $M_n \not\in \X$ and $\ass M_n=\{\p\}$ for all $n$.
We shall do this by induction on $n$.
Assume that we have constructed $M_n$.
Let $f_{n1}, \ldots, f_{n s_n}$ be a system of generators of $\Hom_R(M_n, R/\p)$.
Let $M_{n+1}$ be the kernel of the map $f={}^\t(\begin{smallmatrix}f_{n 1}&\cdots& f_{n s_n}\end{smallmatrix}):M_n \to (R/\p)^{\oplus s_n}$.
As there is a surjection from a direct sum of copies of $R/\p$ to $\image f$ and $\X$ is closed under images and extensions, Lemma \ref{remim0} implies $\image f\in\X$, whence $M_n \notin \X$ implies $M_{n+1} \not\in \X$.
Note that $\ass M_{n+1}=\ass M_n = \{\p \}$.
We thus obtain a desired filtration.

Localizing the filtration at $\p$, we get a descending chain $\cdots \subseteq (M_n)_\p \subseteq (M_{n-1})_\p \subseteq \cdots  \subseteq (M_0)_\p=M_\p$ of $R_\p$-modules.
Since $M_\p$ is an artinian $R_\p$-module, the descending chain stabilizes, i.e., $(M_n)_\p = (M_{n+1})_\p = \cdots$ for some integer $n$.
Then $\Hom_{R_\p}((M_n)_\p, \kappa(\p)) = \sum_{i=1}^{s_n} R_\p f_{n i} =0$, and therefore $(M_n)_\p = 0$.
This contradicts the fact that $\ass M_n = \{\p \}$.
Consequently, the module $M$ belongs to $\X$. 
\end{proof}

The following proposition plays a central role in both this section and the next section.

\begin{prop}\label{aslem2}
Let $\X$ be a subcategory of $\mod R$ closed under images and extensions.
Let $\Phi$ be a set of prime ideals of $R$.
Suppose that $R/\p \in \X$ for all $\p \in \Phi$.
Then $\ass^{-1} \Phi$ is contained in $\X$.
\end{prop}

\begin{proof}
Let $M \in \ass^{-1} \Phi$.
There are prime ideals $\p_1, \ldots, \p_n \in\Phi$ such that $\ass M = \{\p_1, \ldots, \p_n\}$.
Lemma \ref{aslem1} yields an exact sequence $0\to L\to M\to N\to0$ with $\ass L=\{\p_1\}$ and $\ass N=\{\p_2,\dots,\p_n\}$.
By Lemma \ref{18} and induction on $n$, we have $L,N\in\X$.
As $\X$ is closed under extensions, $M$ belongs to $\X$.
\end{proof}

To get applications of the above proposition, we establish a lemma.

\begin{lem}\label{max}
Let $\X$ be a subcategory of $\mod R$.
Then the following two statements hold true.
\begin{enumerate}[\rm(1)]
\item
If $\X$ is $\Hom$-ideal and closed under direct summands, then $R/\m \in \X$ for any $\m\in\Max R\cap\ass \X$.
\item
If $\X$ contains $R$ and is closed under finite direct sums and images, then $\X$ is closed under subobjects.
\end{enumerate}
\end{lem}

\begin{proof}
(1) Choose an $R$-module $X\in \X$ with $\m \in\ass X$.
Then $\Hom_R(R/\m, X)$ is a nonzero module that belongs to $\X$.
This implies that $R/\m$ is in $\X$, since $\Hom_R(R/\m, X)$ is an $R/\m$-vector space.

(2) Let $X$ be a module in $\X$ and $M$ a submodule of $X$.
Composing the inclusion map $M\to X$ with a surjection $R^{\oplus n}\to M$, we get a map $f:R^{\oplus n}\to X$ with $R^{\oplus n},X\in\X$.
We then have $M=\image f\in\X$.
\end{proof}

We obtain a corollary of the above proposition, which gives an affirmative answer to Question \ref{qIKE}.

\begin{cor}\label{d}
Suppose that $R$ has dimension at most one.
Let $\X$ be an IKE-closed subcategory of $\mod R$ such that $\Max R \subseteq \ass \X$.
Then the equality $\X=\ass^{-1}(\ass\X)$ holds.
Therefore, $\X$ is torsionfree.
\end{cor}

\begin{proof}
It is observed from Remark \ref{2}(2) and Lemma \ref{max}(1) that the module $R/\m$ belongs to $\X$ for every $\m \in \Max R$.
By Proposition \ref{aslem2}, it suffices to show that $R/\p\in\X$ for $\p\in\ass\X$.
Take a monomorphism $R/\p\to X$ with $X\in\X$.
Put $H=\Hom_R(R/\p,X)$ and $r=\rank_{R/\p}H>0$.
The $\Hom$-ideal property of $\X$ shows that $H\in\X$.
There is an exact sequence $0\to(R/\p)^{\oplus r}\to H\to C\to0$ of $R/\p$-modules such that $C$ is torsion.
As $\dim R/\p\le1$, the module $C$ has finite length.
Since $\X$ contains $R/\m$ for every $\m \in \Max R$ and is closed under extensions, $C$ belongs to $\X$.
It follows that $(R/\p)^{\oplus r}$ is in $\X$, and so is $R/\p$.
\end{proof}

We get one more corollary which also answers Question \ref{qIKE} in the affirmative; the corollary below particularly says that all the IKE-closed subcategories of $\mod R$ containing $R$ are torsionfree.

\begin{cor}
The assignments $\X \mapsto \ass \X,\,\Phi \mapsto \ass^{-1} \Phi$ give a one-to-one correspondence between
\begin{itemize}
\item
the subcategories of $\mod R$ closed under images and extensions and containing $R$, and
\item
the subsets of $\spec R$ containing $\ass R$.
\end{itemize}
\end{cor}

\begin{proof}
Let $\X$ be a subcategory of $\mod R$ closed under images and extensions and containing $R$, and let $\Phi$ be a subset of $\spec R$ containing $\ass R$.
It suffices to verify $\ass(\ass^{-1}\Phi)\supseteq\Phi$ and $\ass^{-1}(\ass\X)\subseteq\X$.
The former follows from the fact that $\ass R/\p=\{\p\}$.
To show the latter, pick any $\p\in\ass\X$.
Then there is an injective homomorphism $R/\p\to X$ with $X\in\X$.
By Lemma \ref{max}(2) the subcategory $\X$ of $\mod R$ is closed under subobjects, and hence $R/\p$ belongs to $\X$.
Proposition \ref{aslem2} implies $\ass^{-1}(\ass\X)\subseteq\X$.
\end{proof}

\section{A reduction of the question and excellent rings}\label{s4}

In this section, we proceed with the investigation of Question \ref{qIKE}.
We reduce it to a more accessible question.
When the ring $R$ is excellent, we deduce a certain conclusion from the assumption of the new question, and then give a couple of positive answers to the original Question \ref{qIKE}.

Denote by $\tf R$ and $\fl R$ the subcategories of $\mod R$ consisting of torsionfree $R$-modules and consisting of modules of finite length, respectively.
When $R$ is a Cohen--Macaulay local ring of dimension one, $\tf R$ coincides with the subcategory $\cm(R)$ consisting of maximal Cohen--Macaulay $R$-modules.

\begin{rem}\label{3}
\begin{enumerate}[(1)]
\item
One has $\ass^{-1}(\ass R)=\tf R$.
In particular, $\ass^{-1}\{0\}=\tf R$ if $R$ is a domain.
\item
One has $\ass^{-1}(\Max R)=\fl R$.
In particular, $\ass^{-1}\{\m\}=\fl R$ if $(R,\m)$ is local.
\end{enumerate}
\end{rem}

We state a question, which is what we want to consider in this section.

\begin{ques}\label{19}
Let $R$ be a domain.
Let $\X$ be an IKE-closed subcategory of $\mod R$ with $\ass\X=\{0\}$.
Does then $R$ belong to $\X$\,?
\end{ques}

The following proposition says that our original Question \ref{qIKE} is reduced to the above Question \ref{19}.

\begin{prop}\label{21}
Question \ref{19} is equivalent to Question \ref{qIKE}.
\end{prop}

\begin{proof}
The assertion of the proposition follows from (1) and (2) below.

(1) Let $R$ be a domain and $\X$ an IKE-closed subcategory of $\mod R$ with $\ass\X=\{0\}$.
Find $0\ne X\in\X$.
We have $\ass X=\{0\}$.
Assume that Question \ref{qIKE} is affirmative.
Then $\X=\ass^{-1}(\ass\X)$ holds.
We have $\ass R=\{0\}=\ass X\subseteq\ass\X$, which implies $R\in\ass^{-1}(\ass\X)=\X$.
Thus Question \ref{19} is affirmative.

(2) Suppose that Question \ref{19} is affirmative.
We want to show that Question \ref{qIKE} is affirmative as well.
Let $\X$ be an IKE-closed subcategory of $\mod R$.
Fix a prime ideal $\p\in\ass\X$.
According to Proposition \ref{aslem2}, we will be done if we deduce $R/\p\in\X$.
So, assume that $R/\p$ is not in $\X$, and choose $\p$ to be maximal, with respect to the inclusion relation, among the prime ideals $\p'\in\ass\X$ with $R/\p'\notin\X$.
Remark \ref{e} implies that $\X\cap\mod R/\p$ is an IKE-closed subcategory of $\mod R/\p$.
If $\q/\p$ is a nonzero prime ideal of $R/\p$ that belongs to $\ass_{R/\p}(\X\cap\mod R/\p)$, then we have $\p\subsetneq\q\in\ass_R\X$, and $R/\q$ is in $\X\cap\mod R/\p$ by the maximality of $\p$.
We want to deduce $R/\p\in\X$, and then we will have a desired contradiction.
Toward this, replacing $R$ with $R/\p$, we may assume that $R$ is a domain and $\p=0$.
Note that $R/\q\in\X$ for all $0\ne\q\in\ass\X$.
As $0\in\ass\X$, there is an $R$-module $X\in\X$ with $0\in\ass X$.
Putting $\Phi=\ass X\setminus\{0\}$, we have $\ass X=\{0\}\sqcup\Phi$.
Lemma \ref{aslem1} gives rise to an exact sequence $0\to M\to X\to N\to0$ of $R$-modules such that $\ass M=\{0\}$ and $\ass N=\Phi$.
Proposition \ref{aslem2} shows $\ass^{-1}\Phi\subseteq\X$, which implies $N\in\X$.
Since $\X$ is closed under kernels, $M$ belong to $\X$.
Using Remark \ref{3}(1), we get $M\in\X\cap\tf R$.
It is now observed that $\ass(\X\cap\tf R)=\{0\}$.
Since $\X$ and $\tf R$ is IKE-closed, so is the intersection $\X\cap\tf R$.
Our assumption that Question \ref{19} is affirmative yields $R\in\X\cap\tf R$.
Hence $R$ is in $\X$ and the proof is completed.
\end{proof}

It should be much easier to think of Question \ref{19} than Question \ref{qIKE}.
Indeed, the former has a stronger assumption but a weaker conclusion than the latter.
For example, the following observation may be useful to deduce that Question \ref{19} has a positive answer.

\begin{ex}
Let $\X$ be a subcategory of $\mod R$ closed under direct summands and extensions.
Let $\m$ be a maximal ideal of $R$ such that $0\ne\m\in\X$.
Then $R$ belongs to $\X$.

In fact, note that there exists an exact sequence $0\to\m^2\to\m\xrightarrow{a}(R/\m)^{\oplus n}\to0$ of $R$-modules with $n>0$.
Take an epimorphism $b:R^{\oplus n}\to(R/\m)^{\oplus n}$.
The pullback diagram of $a$ and $b$ gives rise to an exact sequence $0\to\m^{\oplus n}\to\m^2\oplus R^{\oplus n}\to\m\to0$.
Since $\m$ and $\m^{\oplus n}$ belong to $\X$, so does $R^{\oplus n}$, and so does $R$.
\end{ex}

Here we prepare an easy lemma.

\begin{lem}\label{remim}
Let $\X$ be a subcategory of $\mod R$ closed under direct sums and images.
\begin{enumerate}[\rm(1)]
\item
One has $IX\in\X$ for all $R$-modules $X\in\X$ and all ideals $I$ of $R$.
\item
Let $S$ be a module-finite $R$-algebra.
Then $S$ is in $\X$ if and only if so is every torsionless $S$-module.
\end{enumerate}
\end{lem}

\begin{proof}
(1) Let $a_1,\dots,a_n$ be a system of generators of the ideal $I$.
Then we have a surjection $f:X^{\oplus n}\to IX$ given by $f({}^\t(x_1,\dots,x_n))=a_1x_1+\cdots+a_nx_n$.
The assertion follows from Lemma \ref{remim0}.

(2) The ``if'' part holds since $S$ is itself a torsionless $S$-module.
Let us show the ``only if'' part.
Assume $S\in\X$, and let $N$ be a torsionless $S$-module.
Then there is an injective homomorphism $N\to S^{\oplus n}$ with $n\ge0$.
Also, there exists a surjective homomorphism $S^{\oplus m}\to N$.
The assertion follows from Lemma \ref{remim0}.
\end{proof}

We state and prove the proposition below.
It would say something close to the affimativity of Question \ref{19} for such a ring $R$ as in it.

\begin{prop} \label{int}
Let $R$ be an excellent domain of characteristic zero.
Let $S$ be the integral closure of $R$.
Let $\X$ be an IKE-closed subcategory of $\mod R$ such that $\ass\X=\{0\}$.
Then every torsionfree $S$-module belongs to $\X$, and so does the module $\Hom_R(N, R)$ for each $S$-module $N$.
\end{prop}

\begin{proof}
First of all, since $R$ is excellent, the $R$-algebra $S$ is module-finite.
Take a nonzero module $X\in \X$.
Set $Y= \Hom_R(S,X)$ and $Z= \Hom_S(Y,Y)$.
We then have $Z=\Hom_S(Y,\Hom_R(S,X))\cong\Hom_R(Y,X)$ by adjointness.
Remark \ref{2}(2) implies $Y,Z\in\X$.
Since $X$ is a torsionfree $R$-module by Remark \ref{3}(1), $Y$ is torsionfree as an $R$-module.
We directly verify that $Y$ is torsionfree as an $S$-module.
It follows from \cite[Proposition A.2 and Corollary A.5]{HL} that $S$ is a direct summand of $Z$.
Therefore, $S$ belongs to $\X$ since $\X$ is closed under direct summands.
It follows by Lemma \ref{remim}(3) that every torsionfree $S$-module is in $\X$.

Fix $N\in\mod S$.
A presentation $S^{\oplus m} \to S^{\oplus n} \to N \to 0$ induces an exact sequence $0 \to \Hom_R(N, R) \to \Hom_R(S, R)^{\oplus n} \to \Hom_R(S, R)^{\oplus m}$.
Since $\Hom_R(S, R)$ is isomorphic to an ideal of $S$, it belongs to $\X$ by Lemma \ref{remim}(3).
As $\X$ is closed under finite direct sums and kernels, we obtain $\Hom_R(N,R)\in\X$.
\end{proof}

Here is a direct application of the above proposition.

\begin{cor}\label{f}
Let $R$ be an excellent normal domain of characteristic zero.
Then there exists no IKE-closed subcategory $\X$ of $\mod R$ with $0\subsetneq\X\subsetneq\tf R$.
\end{cor}

\begin{proof}
Let $\X$ be an IKE-closed subcategory of $\mod R$ with $0\ne\X\subseteq\tf R$.
Then, since $\tf R=\ass^{-1}\{0\}$ by Remark \ref{3}(1), we see that $\ass\X=\{0\}$.
Letting $R=S$ in Proposition \ref{int} yields that the inclusion $\X\supseteq\tf R$ holds.
We obtain the equality $\X=\tf R$, and the assertion of the corollary now follows.
\end{proof}

The $2$-dimensional version of Corollary \ref{d} holds under some mild assumptions.

\begin{cor}
Let $(R,\m,k)$ be a $2$-dimensional excellent normal local domain of characteristic $0$.
Let $\X$ be an IKE-closed subcategory of $\mod R$ with $\m\in\ass\X$.
Then $\X=\ass^{-1}(\ass\X)$ and $\X$ is torsionfree.
\end{cor}

\begin{proof}
According to Proposition \ref{aslem2}, it is enough to prove that $R/\p$ belongs to $\X$ for every $\p\in\ass\X$.
By Remark \ref{2}(2) and Lemma \ref{max}(1) we have that $k\in\X$.
Hence we are done when $\height\p=2$.

Let us consider the case where $\height\p=1$.
Put $\overline\X=\X\cap\mod R/\p$.
Then $\overline\X$ is an IKE-closed subcategory of $\mod R/\p$ with $0\in\ass_{R/\p}\overline\X$.
Note that $\dim R/\p=1$, and that $\m/\p\in\ass_{R/\p}\overline\X$ as $k\in\X$.
It follows from Corollary \ref{d} that $R/\p\in\ass_{R/\p}^{-1}(\ass_{R/\p}\overline\X)=\overline\X$.
Therefore, the module $R/\p$ belongs to $\X$.

Now consider $\height\p=0$.
Then $0=\p\in\ass\X$.
The argument so far proves $R/\q\in\X$ for all $0\ne\q\in\ass\X$.
The latter part of the proof of Proposition \ref{21}(2) shows that $\X\cap\tf R$ is an IKE-closed subcategory of $\mod R$ with $\ass(\X\cap\tf R)=\{0\}$.
Proposition \ref{int} concludes $\X\cap\tf R=\tf R$, which implies $R\in\X$.
\end{proof}

Finally, we study the case where $R$ has dimension one from another point of view than so far.
Let us begin with determining all the torsionfree subcategories of $\mod R$ in the remark below.

\begin{rem}\label{13}
Let $(R,\m)$ be a local domain of dimension one.
Then the torsionfree subcategories of $\mod R$ are $0,\,\fl R,\,\tf R,\,\mod R$.
Indeed, thanks to \cite[Theorem 4.1]{wide}, each of the torsionfree subcategories has the form $\ass^{-1}\Phi$, where $\Phi$ is a subset of $\spec R=\{0,\m\}$.
Note that $\Phi$ is one of the sets $\emptyset,\,\{0\},\,\{\m\},\,\spec R$.
We have $\ass^{-1}\emptyset=0$, $\ass^{-1}\{0\}=\tf R$, $\ass^{-1}\{\m\}=\fl R$ and $\ass^{-1}(\spec R)=\mod R$; see Remark \ref{3}.
\end{rem}

Let $R$ be a reduced ring with total quotient ring $Q$ and integral closure $S$.
The set $R:_QS$ is called the {\em conductor} of $R$.
In what follows, we freely use knowledge of conductors stated in \cite[Chapter 12]{SH}.

We introduce the following notation: for an $R$-module $M$, we denote by $\ext_RM$ the {\em extension closure} of $M$, which is defined to be the smallest subcategory of $\mod R$ that contains $M$ and is closed under direct summands and extensions (note here that we require closedness under direct summands).

Now we prove the following result, which gives a sufficient condition for Question \ref{qIKE} to be affirmative.
This proposition plays an important role in the proofs of the main results of Section \ref{s6}.

\begin{prop}\label{12}
Let $R$ be a one-dimensional excellent henselian local domain.
Let $S$ be the integral closure of $R$.
Let $\c$ be the conductor of $R$.
Then the following statements hold.
\begin{enumerate}[\rm(1)]
\item
Let $\X$ be a Hom-ideal subcategory of $\mod R$ closed under direct summands.
If $\X$ contains a nonzero torsionfree $R$-module $X$, then $\X$ contains the $R$-module $S$.
\item
Suppose that $\ext_R\c=\tf R$ holds.
Then the IKE-closed subcategories of $\mod R$ are $0,\,\fl R,\,\tf R,\,\mod R$.
In particular, Question \ref{qIKE} has a positive answer for $R$.
\end{enumerate}
\end{prop}

\begin{proof}
Since $R$ is an excellent henselian local ring of dimension one, $S$ is a discrete valuation ring.

(1) Put $M=\Hom_R(S,X)$.
It is observed from \cite[Exercise 1.2.27]{BH} that $M$ is nonzero.
The Hom-ideal property of $\X$ shows that $M$ belongs to $\X$.
Since the $R$-module $X$ is maximal Cohen--Macaulay, so is the $S$-module $M$.
Since $S$ is regular, $M$ is $S$-free.
As $\X$ is closed under direct summands, we get $S\in\X$.

(2) Let $\X$ be an IKE-closed subcategory of $\mod R$.
Note that $\spec R=\{0,\m\}$.
If $\ass\X=\emptyset$, then $\X=0$.
If $\m\in\ass\X$, then Corollary \ref{d} implies that $\X$ is torsionfree, and $\X\in\{0,\,\fl R,\,\tf R,\,\mod R\}$ by Remark \ref{13}.
We may assume $\ass\X=\{0\}$.
Then $\X\subseteq\tf R$ by Remark \ref{3}(1).
We have $S\in\X$ by (1), and $\c\in\X$ by Lemma \ref{remim}(3) as $\c$ is an ideal of $S$.
Hence $\X\supseteq\ext_R\c=\tf R$.
We get $\X=\tf R$, which completes the proof of the first assertion.
The second assertion follows from the first and Remark \ref{13}.
\end{proof}

\section{Closedness under direct summands and extensions}\label{s5}

In this section, we consider what modules are built by taking only direct summands and extensions.
Results given in this section are used in our further investigation of Question \ref{qIKE} on IKE-closed subcategories developed in the next section.
We first make elementary observations on extension closures.

\begin{rem}\label{7}
\begin{enumerate}[(1)]
\item
Let $R$ be a local ring with residue field $k$.
Then $\ext_Rk=\fl R$.
Indeed, the inclusion $(\subseteq)$ is obvious.
To show $(\supseteq)$, we take a composition series of each module of finite length.
\item
Let $I$ be an ideal of $R$.
Let $M$ and $N$ be $R/I$-modules.
If $N\in\ext_{R/I}M$, then $N\in\ext_RM$.
This is a consequence of the fact that a short exact sequence of $R/I$-modules is a short exact sequence of $R$-modules, and a direct summand of an $R/I$-module $L$ is a direct summand of the $R$-module $L$.
\end{enumerate}
\end{rem}

We denote by $\tf_0R$ the subcategory of $\mod R$ consisting of torsionfree $R$-modules which are generically free (i.e., locally free on $\Min R$).
When $R$ is a Cohen--Macaulay local ring of dimension one, $\tf_0R$ coincides with the subcategory $\cm_0(R)$ consisting of maximal Cohen--Macaulay $R$-modules that are locally free on the punctured spectrum of $R$.
Below, we establish three lemmas to show the main result of this section.
The first one concerns the extension closures of ideals of a local ring.

\begin{lem}\label{14}
Let $R$ be a local ring with maximal ideal $\m$ and residue field $k$.
\begin{enumerate}[\rm(1)]
\item
Let $I,J$ be ideals of $R$.
If $R/I\in\ext_RR/J$, then $I\in\ext_RJ$.
\item
Suppose that $R$ is a Cohen--Macaulay ring of dimension one.
Let $I$ be an ideal of $R$ which is generically free as an $R$-module.
If $k$ belongs to $\ext_RR/I$, then the equality $\ext_RI=\tf_0R$ holds.
\end{enumerate}
\end{lem}

\begin{proof}
(1) For each $R$-module $M$, denote by $\syz M$ the first syzygy of $M$ in the minimal free resolution of $M$; hence $\syz M$ is uniquely determined up to isomorphism.
Let $\X$ be the subcategory of $\mod R$ consisting of modules $X$ with $\syz X\in\ext_RJ$.
Then it is seen that $\X$ is closed under direct summands and extensions, and contains $R/J$.
Hence $\X$ contains $\ext_RR/J$, which contains $R/I$.
It follows that $I=\syz(R/I)\in\ext_RJ$.

(2) The subcategory $\tf_0R$ is closed under direct summands and extensions, and contains $I$.
Hence $\tf_0R$ contains $\ext_RI$.
Conversely, since $k=R/\m$ belongs to $\ext_RR/I$, it is seen by (1) that $\m$ belongs to $\ext_RI$.
As $R$ is a $1$-dimensional Cohen--Macaulay local ring, there is an equality $\ext_R\m=\tf_0R$ by \cite[Theorem 2.4]{stcm}.
Therefore, $\tf_0R$ is contained in $\ext_RI$.
We conclude that $\tf_0R=\ext_RI$.
\end{proof}

The following lemma may be well-known to experts.
Perhaps it is more usual to show the lemma by using the notion of Ratliff--Rush closures, but we give here an elementary direct proof.

\begin{lem}\label{6}
Let $R$ be a local ring with infinite residue field $k$.
Let $I$ be a proper ideal of $R$.
If the associated graded ring $\gr_I R$ has positive depth, then $I^i:I^j=I^{i-j}$ holds for all integers $i\ge j\ge0$.
\end{lem}

\begin{proof}
We show $I^i:I^j\subseteq I^{i-j}$ by induction on $i-j$.
It is clear if $i-j=0$.
Let $i-j>0$.
The induction hypothesis implies $I^i:I^j\subseteq I^{i-1}:I^j\subseteq I^{i-j-1}$
As $k$ is infinite, we can choose a ($\gr_IR$)-regular element $\overline x\in(\gr_IR)_1=I/I^2$ with $x\in I$; see \cite[Proposition 1.5.12]{BH}.
The injectivity of the multiplication map $I^{i-j-1}/I^{i-j}\xrightarrow{x^j}I^{i-1}/I^i$ deduces $I^{i-j}=(I^i:x^j)\cap I^{i-j-1}$.
As $I^i:I^j\subseteq I^i:x^j$, we get $I^i:I^j\subseteq I^{i-j}$.
\end{proof}

The lemma below is not advanced but plays a crucial role in the proof of our next theorem.
Indeed, it is essential in exploring extension closures to find a short exact sequence as in the proof of the lemma.

\begin{lem}\label{4}
Let $x$ be an element of $R$, and let $I$ be an ideal of $R$.
If $0:_Rx\subseteq xI$, then $R/I\in\ext_RR/xI$.
\end{lem}

\begin{proof}
It suffices to show that $0\to R/xI\xrightarrow{f}R/I\oplus R/x^2I\xrightarrow{g}R/xI\to0$ is an exact sequence, where $f(\overline a)=\binom{\overline a}{\overline{xa}}$ and $g(\binom{\overline a}{\overline b})=\overline{b-xa}$ for $a,b\in R$.
It is easy to see that $f,g$ are well-defined homomorphisms, $g$ is surjective and $gf=0$.
If $b-xa\in xI$, then $b-xa=xc$ for some $c\in I$ and $\binom{\overline a}{\overline b}=\binom{\overline{a+c}}{\overline{x(a+c)}}$.
Hence the equality $\image f=\kernel g$ holds.
Suppose $xa\in x^2I$.
Then $xa=x^2d$ for some $d\in I$ and $a-xd\in0:_Rx$.
The assumption $0:_Rx\subseteq xI$ implies $a-xd\in xI$, and $a\in xI$.
This shows that $f$ is injective.
\end{proof}

Now we can prove the main result of this section, which is the following theorem.
This is not only used in the proof of one of the main results of the next section on IKE-closed subcategories, but also of independent interest as a result purely about subcategories closed under direct summands and extensions.

\begin{thm}\label{8}
Let $(R,\m,k)$ be a Cohen--Macaulay local ring of dimension one with $k$ infinite.
Suppose that $\gr_\m R$ has positive depth.
Then $\ext_RR/\m^i=\fl R$ and $\ext_R\m^i=\tf_0R$ for all positive integers $i$.
\end{thm}

\begin{proof}
The second assertion is a consequence of the first and Lemma \ref{14}(2).
In what follows, we prove the first assertion.
As $k$ is infinite and $\dim R=1$, we can choose a system of parameters $x$ of $R$ such that $(x)$ is a reduction of $\m$; see \cite[Corollary 4.6.10]{BH}.
There is an integer $n>0$ such that $\m^{n+1}=x\m^n$.
Since $k$ is infinite and $\depth(\gr_\m R)>0$, we have $\m^p:\m^q=\m^{p-q}$ for all integers $p\ge q\ge0$ by Lemma \ref{6}.

Fix an integer $i$ with $1\le i\le n$.
The ideal $\m^i$ contains $x\m^{i-1}$, and there is an ideal $I\subseteq\m^i$ such that $\m^i=x\m^{i-1}+I$.
Choose such an ideal $I$ so that the minimal number of generators $\mu(I)$ of $I$ is minimum.

We claim that $0:_{R/I}x\subseteq x(\m^{i-1}/I)$.
In fact, assume that it is not true.
We can choose an element $y\in R$ with $\overline y\in0:_{R/I}x$ and $\overline y\notin x(\m^{i-1}/I)$.
Then $xy\in I$ and $y\notin x\m^{i-1}+I=\m^i$.
There are implications
$$
xy\in\m I
\ \Rightarrow\  xy\in\m\m^i=\m^{i+1}
\ \Rightarrow\  xy\m^{n-i}\subseteq\m^{n+1}=x\m^n
\ \Rightarrow\  y\m^{n-i}\subseteq\m^n
\ \Rightarrow\  y\in\m^n:\m^{n-i}=\m^i,
$$
where the third implication holds since $x$ is $R$-regular, and the last equality holds since $n\ge n-i\ge0$.
As $y\notin\m^i$, we get $xy\notin\m I$.
Hence $xy$ is part of a minimal system of generators of $I$.
If $xy\in x\m^{i-1}$, then there exists an ideal $I'$ with $\mu(I')<\mu(I)$ such that $\m^i=x\m^{i-1}+I'$, which contradicts the choice of $I$.
We have $xy\notin x\m^{i-1}$, and $y\notin\m^{i-1}$.
Similarly as above, we see that there are implications
$$
xy\in I
\ \Rightarrow\  xy\in\m^i
\ \Rightarrow\  xy\m^{n-i+1}\subseteq\m^{n+1}=x\m^n
\ \Rightarrow\  y\m^{n-i+1}\subseteq\m^n
\ \Rightarrow\  y\in\m^n:\m^{n-i+1}=\m^{i-1}.
$$
Recall that we have $xy\in I$ and $y\notin\m^{i-1}$.
The implications give a contradiction, and the claim follows.

The above claim enables us to apply Lemma \ref{4} to the ring $R/I$ and the ideal $\m^{i-1}/I$ to get $R/\m^{i-1}\in\ext_{R/I}R/\m^i$ (since $x\m^{i-1}+I=\m^i$).
By Remark \ref{7}(2), we get $\ext_RR/\m^{i-1}\subseteq\ext_RR/\m^i$ for all $1\le i\le n$.
Using Remark \ref{7}(1), we see that $\ext_RR/\m^i=\fl R$ for all $1\le i\le n$.
Recall that $n$ is a positive integer such that $\m^{n+1}=x\m^n$.
Multiplying this equality by $\m^{n'-n}$, we have $\m^{n'+1}=x\m^{n'}$ for all integers $n'\ge n$.
Replacing $n$ with $n'$ in the above argument, we observe that $\ext_RR/\m^i=\fl R$ for all integers $i>0$.
\end{proof}

\section{Numerical semigroup rings}\label{s6}

In this section, we focus on the case where $R$ is a numerical semigroup ring.
We apply results which we have obtained in the previous sections, especially Proposition \ref{12} and Theorem \ref{8}, and figure out certain cases where our Question \ref{qIKE} has a positive answer.

Recall that a {\em numerical semigroup} is by definition a subsemigroup $H$ of the additive semigroup $\N$ such that $0\in H$ and $\N\setminus H$ is finite.
In this section, we investigate IKE-closed subcategories of $\mod R$ in the case where $R$ is a (completed) numerical semigroup ring to consider our Question \ref{qIKE}.

In what follows, we freely use knowledge of numerical semigroups stated in \cite[Page 178]{BH}.
We begin with introducing a numerical semigroup defined by consecutive integers and computing its conductor.

\begin{dfn}
For positive integers $a$ and $r$, we put
$$
H_{a,r}= \langle a, a+1, \dots, a+r \rangle
=\{c_0a+c_1(a+1)+\cdots+c_r(a+r)\mid c_0,c_1,\dots,c_r\in\N\}.
$$
As $\gcd(a,a+1)=1$, the set $\N\setminus\langle a,a+1\rangle$ is finite, and so is $\N\setminus H_{a,r}$.
Thus $H_{a,r}$ is a numerical semigroup.
\end{dfn}

\begin{lem} \label{numlem}
Let $a,r$ be positive integers.
Let $c$ be the conductor of the numerical semigroup $H_{a,r}$.
\begin{enumerate}[\rm(1)]
\item	
For any integers $n \ge 1$ and $0 \le j \le nr$, one has $na + j \in H_{a,r}$.
\item
If $a \le ur + 1$ for an integer $u$, then $ua \ge c$ (i.e., $b \in H_{a,r}$ for any integer $b \ge ua$).
\item
If $a > u r + 1$, then $ua+j \not\in H_{a,r}$ for any integer $ur < j \le a-1$.
In particular, $c > ua + (a-1)$.
\end{enumerate}	
\end{lem}

\begin{proof}
(1) We have $j = qr + i$ for some integers $q$ and $0 \le i \le r-1$.	
Then $0 \le q \le n$ since $0 \le j \le nr$.
If $i = 0$ (i.e., $q= n$), then we get $na + j = n(a+r) \in H_{a, r}$.
If $i > 0$ (i.e., $q < n$), then we also get $na + qr + i = (n-q-1)a + q(a+r)+ (a+i) \in H_{a,r}$.

(2) Note that $a \le ur + 1$ if and only if $ua + ur \ge ua + (a-1)$.
Then (1) and the assumption yield that $ua + j \in H_{a,r}$ for all $0 \le j \le a-1$.
Since any integer $b \ge ua$ can be written as $b = ma + j$ for some $m \ge n$ and $0 \le j \le a-1$, we get $b = (m-n)a + (na + j) \in H_{a,r}$.

(3) Similarly as above, $a > ur + 1$ if and only if $ua + ur < ua + (a-1)$.
Assume $ua+j \in H_{a,r}$ and we seek a contradiction.
Then $ua+j = \sum_{i=0}^r l_i (a+i)$ for some integers $l_i \ge 0$.
Since $j \le a-1$, we have $(u+1)a > ua+j = \sum_{i=0}^r l_i (a+i) \ge (\sum_{i=0}^r l_i) a$.
Therefore, $\sum_{i=0}^r l_i \le u$ and hence $\sum_{i=0}^r l_i \cdot i \le ur < j$ hold.
Then we conclude $u = \sum_{i=0}^r l_i$ and $j = \sum_{i=0}^r l_i \cdot i$, which is a desired contradiction.
\end{proof}

\begin{prop} \label{num1}
Let $c$ be the conductor of $H_{a,r}$, where $a, r \ge 1$.
Then one has $c=\lceil \frac{a-1}{r} \rceil \cdot a$.
\end{prop}

\begin{proof}
Set $u=\lceil \frac{a-1}{r} \rceil$.
Then there are inequalities $(u-1)a + 1 < a \le ua +1$.
Therefore, it follows from (2) and (3) of Lemma \ref{numlem} that the desired equality $c = ua$ holds true.
\end{proof}

Let $H$ be a numerical semigroup.
Let $B=k[t]$ be a polynomial ring over a field $k$.
Take the subring $A=k[t^h|h\in H]$ of $B$ and the ideal $I=(t^h|h\in H)$ of $A$.
We denote by $k[\![H]\!]$ the $I$-adic completion of $A$, and call it the {\em numerical semigroup ring} of $H$.
Note that the formal power series ring $k[\![t]\!]$ is the integral closure of $k[\![H]\!]$.
We establish two lemmas on numerical semigroup rings to deduce our next proposition.

\begin{lem}\label{20}
Let $R=k[\![H]\!]$ be a numerical semigroup ring over a field $k$ with integral closure $S=k[\![t]\!]$.
Denote by $\m$ the maximal ideal of $R$.
Then the following two statements hold true.
\begin{enumerate}[\rm(1)]
\item
For an ideal $I$ of $R$, one has $IS=I$ if and only if $I= t^aS$ for some integer $a\ge0$.
\item
If $H=\langle a_1, a_2, \ldots, a_r \rangle$ with $a_1 < a_2 < \cdots < a_r$, then $\m^nS= t^{na_1}S$.
\end{enumerate}
\end{lem}

\begin{proof}
It is straightforward to verify the second assertion.
In what follows, we show the first assertion.
The ``if" part is clear.
To show the ``only if'' part, suppose $IS=I$.
Take $a$ to be the minimum integer $i$ with $t^i \in I$.
For any integer $n\ge0$, we have $t^{a+n} = t^a \cdot t^n \in IS=I$.
This means that $I= t^aS$ holds.
\end{proof}

\begin{lem} \label{num2}
Let $a$ and $r$ be positive integers.
Let $R=k[\![H_{a,r}]\!]$ be the numerical semigroup ring over a field $k$ with integral closure $S=k[\![t]\!]$.
Let $\m$ be the maximal ideal of $R$.
Then the equality $\m^nS= \m^n$ (or equivalently, $\m^n = t^{na}S$ by Lemma \ref{20}(2)) holds if and only if one has the inequality $n \ge\lceil \frac{a-1}{r} \rceil$.
\end{lem}

\begin{proof}
Set $u=\lceil \frac{a-1}{r} \rceil$.
Assume $n < u$.
Then $nr + 1 < a$ since $n \le u-1 <\lceil \frac{a-1}{r}\rceil$.
By Lemma \ref{numlem}(3), we have $na + a -1 \not\in H$.
Therefore, we get $t^{na+a-1} \in t^{na}S\setminus R \subseteq t^{na}S\setminus \m^n$.
Hence $\m^n \ne t^{na}S$.

Next consider the case of $n \ge u$.
In this case, $nr + 1 \ge a$ holds.
For any integer $0 \le j \le a-1 \le nr$, $na + j$ is the sum of $n$ elements of $H$ by the proof of Lemma \ref{numlem}(1).
Since every integer $b \ge na$ is of the form $b = ma + j$ for some $m \ge n$ and $0 \le j \le a-1$, it is the sum of $m$ elements of $H$.
Therefore, we conclude $t^b \in \m^m$ for all integers $b \ge na$.
Hence the equality $\m^n = t^{na}S$ holds.
\end{proof}

Now the proposition below is deduced; it is a direct consequence of Proposition \ref{num1} and Lemma \ref{num2}.
This proposition especially says that the conductor of the numerical semigroup ring of $H_{a,r}$ is given by a power of the maximal ideal, which plays an essential role in the proof of our next theorem.

\begin{prop}\label{cond}
Let $a,r$ be positive integers.
Let $R=k[\![H_{a,r}]\!]$ be the numerical semigroup ring over a field $k$ with integral closure $S=k[\![t]\!]$.
Let $\c$ be the conductor of $R$.
Then $\c=\m^u = t^{ua}S$, where $u=\lceil\frac{a-1}{r}\rceil$.
\end{prop}

The following theorem is one of the main results of this section, which yields complete classification of the IKE-closed subcategories of the module category of the numerical semigroup ring of $H_{a,r}$.

\begin{thm}\label{16}
Let $a\ge2,\,r\ge1$ and $H=H_{a,r}=\langle a,a+1,\dots,a+r\rangle$.
Let $R=k[\![H]\!]$, where $k$ is an infinite field.
Let $\c$ be the conductor of $R$.
Then the following statements hold.
\begin{enumerate}[\rm(1)]
\item
There are equalities $\ext_RR/\c=\fl R$ and $\ext_R\c=\tf R$.
\item
The IKE-closed subcategories of $\mod R$ are $0,\,\fl R,\,\tf R,\,\mod R$.
In particular, Question \ref{qIKE} has an affirmative answer for $R$.
\end{enumerate}
\end{thm}

\begin{proof}
(1) We may assume $r\le a-1$.
Note then that $\gcd(a,a+1,\dots,a+r)=1$.
Proposition \ref{cond} implies $\c=\m^t$ for some $t>0$.
As $a,a+1,\dots,a+r$ is an arithmetic sequence, $\gr_\m R$ is Cohen--Macaulay by \cite[Proposition 1.1]{MT}.
Hence $\depth(\gr_\m R)=\dim(\gr_\m R)=\dim R=1>0$.
It follows from Theorem \ref{8} that $\ext_RR/\m^t=\fl R$ and $\ext_R\m^t=\tf_0R=\tf R$.
Now the assertion follows.

(2) The assertion is an immediate consequence of (1) and Proposition \ref{12}(2).
\end{proof}

To obtain one more theorem, we prove the general proposition below concerning extension closures.

\begin{prop}\label{5}
Let $A$ be a ring.
Let $R=A[x,y]/(x^a\pm y^b)$ be a quotient of a polynomial ring over $A$.
Let $I=(x^{a_1}y^{b_1},\dots,x^{a_n}y^{b_n})$ be an ideal of $R$, where $a>a_1>\cdots>a_n=0$ and $0=b_1<\cdots<b_n<b$ with $n\ge2$.
Then the $R$-module $R/(x,y)$ belongs to $\ext_RR/I$.
\end{prop}

\begin{proof}
We have $I=(x^{a_1},x^{a_2}y^{b_2},\dots,x^{a_{n-1}}y^{b_{n-1}},y^{b_n})$.
Since $a_{n-1}>a_n=0$, we can define the ideal $J=(x^{a_1-1},x^{a_2-1}y^{b_2},\dots,x^{a_{n-1}-1}y^{b_{n-1}})$ of $R$ and have $I=xJ+(y^{b_n})$.
There is an isomorphism $R/(y^{b_n})\cong A[x,y]/(x^a\pm y^b,y^{b_n})=A[x,y]/(x^a,y^{b_n})$ since $b>b_n$.
Hence $0:_{R/(y^{b_n})}x=x^{a-1}(R/(y^{b_n}))$, which is contained in $I(R/(y^{b_n}))= xJ(R/(y^{b_n}))$ as $a-1\ge a_1$.
It follows from Lemma \ref{4} that $R/(J+(y^{b_n}))\in\ext_{R/(y^{b_n})}R/I$.
By Remark \ref{7}(2), we have $R/(J+(y^{b_n}))\in\ext_RR/I$.
Note that
$$
J+(y^{b_n})=(x^{a_1-1},x^{a_2-1}y^{b_2},\dots,x^{a_{n-1}-1}y^{b_{n-1}},y^{b_n}),\qquad a>a_1-1>a_2-1>\cdots>a_{n-1}-1.
$$
If either ($n=2$ and $a_1-1>1$) or ($n\ge3$ and $a_{n-1}-1>0$), then we can apply the same argument to get $R/K\in\ext_RR/(J+(y^{b_n}))\subseteq\ext_RR/I$, where $K:=(x^{a_1-2},x^{a_2-2}y^{b_2},\dots,x^{a_{n-1}-2}y^{b_{n-1}},y^{b_n})$.
Iterating this procedure, we obtain $R/L\in\ext_RR/I$, where $L=(x,y^{b_2})$ if $n=2$, and
$$
\begin{array}{l}
L=(x^{a_1-a_{n-1}},x^{a_2-a_{n-1}}y^{b_2},\dots,x^{a_{n-2}-a_{n-1}}y^{b_{n-2}},y^{b_{n-1}},y^{b_n})\\
\phantom{L}=(x^{a_1-a_{n-1}},x^{a_2-a_{n-1}}y^{b_2},\dots,x^{a_{n-2}-a_{n-1}}y^{b_{n-2}},y^{b_{n-1}})
\end{array}
$$
if $n\ge3$; the last equality holds since $b_{n-1}<b_n$.

When $n=2$, we replace $x$ with $y$ in the above argument on the ideal $I=(x^{a_1},y^{b_2})$.
Applying it to the ideal $L=(x,y^{b_2})$ and using the assumption $b-1\ge b_2$, we obtain $R/(x,y)\in\ext_RR/L\subseteq\ext_RR/I$.

Let us consider the case where $n\ge3$.
We then have $a>a_1-a_{n-1}>a_2-a_{n-1}>\cdots>a_{n-2}-a_{n-1}>0$.
Applying the above argument on $I$ to $L$, we see that $R/M\in\ext_RR/L\subseteq\ext_RR/I$, where $M=(x^{a_1-a_{n-2}},x^{a_2-a_{n-2}}y^{b_2},\dots,x^{a_{n-3}-a_{n-2}}y^{b_{n-3}},y^{b_{n-2}})$.
Repeating this, we finally obtain $R/(x^{a_1-a_2},y^{b_n})\in\ext_RR/I$.
The above argument in the case $n=2$ deduces the containment $R/(x,y)\in\ext_RR/I$.
\end{proof}

Now we can prove the following theorem, which is another main result of this section.
This theorem completely classifies the IKE-closed subcategories of the module category of the numerical semigroup ring of a numerical semigroup minimally generated by two elements.

\begin{thm}\label{17}
Let $a>b>0$ be integers with $\gcd(a,b)=1$.
Let $H=\langle a,b\rangle$ be a numerical semigroup, and let $R=k[\![H]\!]$ be the numerical semigroup ring of $H$ over a field $k$.
Let $\m$ be the maximal ideal of $R$ and $\c$ the conductor of $R$.
Then the following statements hold.
\begin{enumerate}[\rm(1)]
\item
There are equalities $\ext_RR/\c=\fl R$ and $\ext_R\c=\tf R$.
\item
The IKE-closed subcategories of $\mod R$ are $0,\,\fl R,\,\tf R,\,\mod R$.
In particular, Question \ref{qIKE} has an affirmative answer for $R$.
\end{enumerate}
\end{thm}

\begin{proof}
(1) Let $S=k[\![t]\!]$ be a formal power series ring, which is equal to the integral closure of $R=k[\![t^a,t^b]\!]$.
Let $c$ be the conductor of the numerical semigroup $H$.
Then $c=(a-1)(b-1)$ and $\c=t^cS=(t^c,t^{c+1},\dots,t^{c+b-1})R$.
We identify $R$ with the quotient $k[\![x,y]\!]/(x^a-y^b)$ of a formal power series ring, so that $x=t^b$ and $y=t^a$ in $R$.
Take any integer $n$ with $c\le n\le c+b-1$.
Then there exist integers $p,q\ge0$ such that $n=ap+bq$.
Hence $t^n=(t^a)^p(t^b)^q=x^qy^p$.
Note that
$$
n\le c+b-1=a(b-1)+b\cdot0,\qquad
n\le c+b-1\le c+a-1=a\cdot0+b(a-1).
$$
We see that $0\le p\le b-1$ and $0\le q\le a-1$.

We claim that one can choose integers $a>a_1>\cdots>a_n=0$ and $0=b_1<\cdots<b_n<b$ with $n\ge2$ such that $\c=(x^{a_1}y^{b_1},\dots,x^{a_n}y^{b_n})R$.
Indeed, since $\c$ is a monomial ideal, there is a minimal system of generators $\{x^{a_i}y^{b_i}\}_{i=1}^n$.
If $b_{i} = b_{j}$ and $a_i \ge a_j$, then $x^{a_j}y^{b_{j}}$ divides $x^{a_{i}}y^{b_{i}}$, contradicting the minimality of $\{x^{a_i}y^{b_i}\}_{i=1}^n$.
Arranging the order of $\{b_i\}_{i=1}^n$, we may assume $b_1 < \cdots < b_n < b$.
If $a_{i+1} \ge a_{i}$, then $x^{a_i}y^{b_{i}}$ divides $x^{a_{i+1}}y^{b_{i+1}}$, again giving a contradiction.
Therefore, there are integers $a> a_1 > \cdots > a_n$ and $b_1 < \cdots < b_n < b$ such that $\c=(x^{a_1}y^{b_1},\dots,x^{a_n}y^{b_n})R$.
Moreover, we have $a_n = b_1 = 0$ as $\c$ is $\m$-primary.

Applying Proposition \ref{5} to $A=k,\,I=(x^{a_1}y^{b_1},\dots,x^{a_n}y^{b_n})$ and taking the $(x,y)$-adic completion, we see that $k\in\ext_RR/\c$.
The assertion follows from Remark \ref{7}(1) and Lemma \ref{14}(2).

(2) The assertion is an immediate consequence of (1) and Proposition \ref{12}(2).
\end{proof}

\begin{ac}
The authors thank Kazuho Ozeki for helpful discussion with the fourth author about associated graded rings in April, 2020.
The authors also thank Haruhisa Enomoto for useful information on IKE-closed subcategories.
\end{ac}


\begin{thebibliography}{9}
\bibitem{AB}
{\sc M. Auslander; M. Bridger}, Stable module theory, Memoirs of the American Mathematical Society {\bf 94}, {\em American Mathematical Society, Providence, R.I.}, 1969.
\bibitem{BH}
{\sc W. Bruns; J. Herzog}, Cohen--Macaulay rings, revised edition, Cambridge Studies in Advanced Mathematics {\bf 39}, {\it Cambridge University Press, Cambridge}, 1998.
\bibitem{E}
{\sc H. Enomoto}, Rigid modules and ICE-closed subcategories in quiver representations, {\em J. Algebra} {\bf 594} (2022), 364--388.
\bibitem{G}
{\sc P. Gabriel}, Des cat\'egories ab\'eliennes, {\em Bull. Soc. Math. France} {\bf 90} (1962), 323--448.
\bibitem{GW}
{\sc S. Goto; K.-i. Watanabe}, Commutative algebra (Japanese), {\em Nippon Hyoron Sha Co.,Ltd. Publishers}, 2011.
\bibitem{HL}
{\sc C. Huneke; G. J. Leuschke}, On a conjecture of Auslander and Reiten, {\em J. Algebra} {\bf 275} (2004), 781--790.
\bibitem{SH}
{\sc C. Huneke; I. Swanson}, Integral closure of ideals, rings, and modules, London Mathematical Society Lecture Note Series {\bf 336}, {\em Cambridge University Press, Cambridge}, 2006.
\bibitem{MT}
{\sc S. Molinelli; G. Tamone}, On the Hilbert function of certain rings of monomial curves, {\em J. Pure Appl. Algebra} {\bf 101} (1995), no. 2, 191--206.
\bibitem{SW}
{\sc D. Stanley; B. Wang}, Classifying subcategories of finitely generated modules over a Noetherian ring, {\em J. Pure Appl. Algebra} {\bf 215} (2011), no. 11, 2684--2693.
\bibitem{wide}
{\sc R. Takahashi}, Classifying subcategories of modules over a commutative Noetherian ring, {\em J. Lond. Math. Soc. (2)} {\bf 78} (2008), no. 3, 767--782.
\bibitem{stcm}
{\sc R. Takahashi}, Classifying thick subcategories of the stable category of Cohen--Macaulay modules, {\em Adv. Math.} {\bf 225} (2010), no. 4, 2076--2116.
\end{thebibliography}
\end{document}